\documentclass[11pt,twoside,reqno]{amsart}
\usepackage[latin1]{inputenc}
\usepackage[italian,english]{babel}
\usepackage{amssymb,latexsym}
\usepackage{amsmath,amsfonts,amsthm}
\usepackage{paralist}

\numberwithin{equation}{section}

\newcommand{\R}{\mathbb{R}}
\newcommand{\N}{\mathbb{N}}

\newcommand{\Z}{\mathbb{Z}}
\newcommand{\s}{\sharp}

\DeclareMathOperator{\dive}{div}

\newtheorem{lem}{Lemma}
\newtheorem{thm}{Theorem}

\newtheorem{defn}{Definition} 
\theoremstyle{remark}
\newtheorem{remark}{Remark}

\begin{document}

\title[periodic solutions for a fractional Schr\"odinger equation]{periodic solutions for a superlinear fractional problem without the Ambrosetti-Rabinowitz condition}
\author[Vincenzo Ambrosio]{Vincenzo Ambrosio}

\address{%
Dipartimento di Matematica e Applicazioni\\ Universit\`a degli Studi "Federico II" di Napoli\\
via Cinthia, 80126 Napoli, Italy}

\email{vincenzo.ambrosio2@unina.it}


\keywords{Periodic solutions, Extension method, Linking Theorem, Cerami condition}

\date{}
\begin{abstract}
The purpose of this paper is to study $T$-periodic solutions to 
\begin{equation}\label{NP}
\left\{
\begin{array}{ll}
[(-\Delta_{x}+m^{2})^{s}-m^{2s}]u=f(x,u) &\mbox{ in } (0,T)^{N}   \\
u(x+Te_{i})=u(x)    &\mbox{ for all } x \in \R^{N}, i=1, \dots, N
\end{array}
\right.
\end{equation}
where $s\in (0,1)$, $N>2s$, $T>0$, $m> 0$ and $f(x,u)$ is a continuous function, $T$-periodic in $x$ and satisfying a suitable growth assumption weaker than the Ambrosetti-Rabinowitz condition. \\
The nonlocal operator $(-\Delta_{x}+m^{2})^{s}$ can be realized as the Dirichlet to Neumann map for a degenerate  elliptic problem posed on the half-cylinder $\mathcal{S}_{T}=(0,T)^{N}\times (0,\infty)$. By using a variant of the Linking Theorem, we show that the extended problem in $\mathcal{S}_{T}$ admits a nontrivial solution $v(x,\xi)$ which is $T$-periodic in $x$. Moreover, by a procedure of limit as $m\rightarrow 0$, we also prove the existence of a nontrivial solution to (\ref{NP}) with $m=0$.
\end{abstract}


\maketitle

\section{Introduction}
\noindent
The aim of this paper is to investigate the existence of periodic solutions to the problem 
\begin{equation}\label{P}
\left\{
\begin{array}{ll}
[(-\Delta_{x}+m^{2})^{s}-m^{2s}]u=f(x,u) &\mbox{ in } (0,T)^{N}   \\
u(x+Te_{i})=u(x)    &\mbox{ for all } x \in \R^{N}, \, i=1, \dots, N
\end{array},
\right.
\end{equation}
where $s\in (0,1)$, $N >2s$, $m\geq 0$, $(e_{i})$ is the canonical basis in $\R^{N}$ and the nonlinearity $f:\R^{N}\times \R \rightarrow \R$ is a superlinear continuous function.
The operator $\displaystyle{(-\Delta_{x}+m^{2})^{s}}$ appearing in (\ref{P}), is defined through the spectral decomposition, by using the powers of the eigenvalues of $-\Delta_{x}+m^{2}$ with periodic boundary conditions.\\
Let $u\in\mathcal{C}^{\infty}_{T}(\R^{N})$, that is $u$ is infinitely differentiable in $\R^{N}$ and $T$-periodic in each variable.
Then $u$ has a Fourier series expansion:
$$
u(x)=\sum_{k\in \Z^{N}} c_{k} \frac{e^{\imath \omega k\cdot x}}{{\sqrt{T^{N}}}} \quad (x\in \R^{N})
$$
where 
$$
 \omega=\frac{2\pi}{T} \,\mbox{ and } \, c_{k}=\frac{1}{\sqrt{T^{N}}} \int_{(0,T)^{N}} u(x)e^{- \imath \omega k \cdot x}dx \quad (k\in \Z^{N})
$$
are the Fourier coefficients of $u$.
The operator $(-\Delta_{x}+m^{2})^{s}$ is defined by setting 
\begin{equation*}\label{nfrls}
(-\Delta_{x}+m^{2})^{s} \,u=\sum_{k\in \Z^{N}} c_{k} (\omega^{2}|k|^{2}+m^{2})^{s} \, \frac{e^{\imath \omega k\cdot x}}{{\sqrt{T^{N}}}}.
\end{equation*}

For  $\displaystyle{u=\sum_{k\in \Z^{N}} c_{k} \frac{e^{\imath \omega k\cdot x}}{{\sqrt{T^{N}}}}}$ and $\displaystyle{v=\sum_{k\in \Z^{N}} d_{k} \frac{e^{\imath \omega k\cdot x}}{{\sqrt{T^{N}}}}}$, we have that 
$$
\mathcal{Q}(u,v)=\sum_{k\in \Z^{N}} (\omega^{2}|k|^{2}+m^{2})^{s} c_{k} \bar{d}_{k}
$$
can be extended by density to a quadratic form on the Hilbert space 
$$
\mathbb{H}^{s}_{m,T}=\Bigl\{u=\sum_{k\in \Z^{N}} c_{k} \frac{e^{\imath \omega k\cdot x}}{{\sqrt{T^{N}}}}\in L^{2}(0,T)^{N}: \sum_{k\in \Z^{N}} (\omega^{2}|k|^{2}+m^{2})^{s} \, |c_{k}|^{2}<\infty \Bigr\}
$$
endowed with the norm
$$
|u|_{\mathbb{H}^{s}_{m,T}}=\sqrt{\sum_{k\in \Z^{N}} (\omega^{2}|k|^{2}+m^{2})^{s} |c_{k}|^{2}}.
$$
When $m=1$ we set $\displaystyle{\mathbb{H}^{s}_{T}=\mathbb{H}^{s}_{1,T}}$. \\
We would remind that in the last decade a great attention has been devoted to the study of fractional Sobolev spaces and 
non-local operators. In fact such operators arise in several fields such as optimization, finance, phase transitions, stratified materials, anomalous diffusion, crystal dislocation, flame propagation, conservation laws, ultra-relativistic limits of quantum mechanics, quasi-geostrophic flows, minimal surfaces and water waves; see for instance \cite{Applebaum}, \cite{BKW}, \cite{Cabsolmor}, \cite{CVal}, \cite{CT}, \cite{FJL}, \cite{LY1}, \cite{Silvestre}, \cite{SirVal}, \cite{Stoker} and references therein.

\smallskip

\noindent
Now we state the assumptions on the nonlinear term $f$ in (\ref{P}). Along the paper, we will suppose that $f:\R^{N}\times \R\rightarrow \R$ verifies the following conditions:
\begin{compactenum}[($f1$)]
\item $f(x,t)$ is $T$-periodic in $x \in \R^{N}$, that is $f(x+Te_{i},t)=f(x,t)$ for any $x\in \R^{N}$, $t\in \R$ and $i=1,\dots, N$;
\item $f$ is continuous in $\R^{N+1}$; 
\item $f(x,t)=o(t)$ as $t \rightarrow 0$ uniformly in $x\in \R^{N}$;
\item there exists $\displaystyle{1<p<2^{\s}_{s}-1=\frac{2N}{N-2s}-1}$ and $C>0$ such that
$$
|f(x,t)|\leq C(1+|t|^{p})
$$
for any $x\in \R^{N}$ and $t\in \R$;
\item $\displaystyle{\lim_{|t|\rightarrow \infty} \frac{F(x,t)}{|t|^{2}}=+\infty}$ uniformly for any $x\in \R^{N}$.
Here $\displaystyle{F(x,t)=\int_{0}^{t} f(x,\tau) d\tau}$;
\item There exists $\gamma\geq 1$ such that for any $x\in \R^{N}$
$$
G(x,\theta t)\leq \gamma \, G(x,t) \mbox{ for any } x\in \R^{N},  t\in \R \mbox{ and } \theta\in [0,1],
$$
where $\displaystyle{G(x,t):=f(x,t)t-2F(x,t)}$.
\end{compactenum}

\noindent
Our first main result is the following 
\begin{thm}\label{thm1}
Let $m>0$ and let us assume that $f:\R^{N+1} \rightarrow \R$ is a function satisfying the assumptions $(f1)-(f6)$.
Then there exists a solution $u\in \mathbb{H}^{s}_{m,T}$ to (\ref{P}).
In particular $u\in \mathcal{C}^{0,\alpha}([0,T]^{N})$ for some $\alpha \in (0,1)$.
\end{thm}


\noindent
To study the problem (\ref{P}), we will make use of a Caffarelli-Silvestre type-extension in the periodic framework (see \cite{A2, A3}). 
This method, which has been originally introduced in \cite{CafSil} to investigate the fractional Laplacian in $\R^{N}$, is very useful because it allows us to reformulate the non-local problem (\ref{P}) in terms of a local degenerate elliptic problem with a Neumann boundary condition in one dimension higher.



More precisely, for $u\in \mathbb{H}^{s}_{m,T}$ one considers the problem
\begin{equation*}
\left\{
\begin{array}{ll}
-\dive(\xi^{1-2s} \nabla v)+m^{2}\xi^{1-2s}v =0 &\mbox{ in }\mathcal{S}_{T}:=(0,T)^{N} \times (0,\infty)  \\
v_{| {\{x_{i}=0\}}}= v_{| {\{x_{i}=T\}}} & \mbox{ on } \partial_{L}\mathcal{S}_{T}:=\partial (0,T)^{N} \times [0,\infty) \\
v(x,0)=u(x)  &\mbox{ on }\partial^{0}\mathcal{S}_{T}:=(0,T)^{N} \times \{0\}
\end{array},
\right.
\end{equation*}
from where the operator $(-\Delta_{x}+m^{2})^{s}$ is obtained as 
$$
-\lim_{\xi\rightarrow 0} \xi^{1-2s} \frac{\partial v}{\partial \xi}(x,\xi) = \kappa_{s} (-\Delta_{x} + m^{2})^{s} u(x), 
$$
in distributional sense and $\displaystyle{\kappa_{s}:= 2^{1-2s} \frac{\Gamma(1-s)}{\Gamma(s)}}$.\\
We will exploit this fact and we will look for solutions to
\begin{equation}\label{R}
\left\{
\begin{array}{ll}
-\dive(\xi^{1-2s} \nabla v)+m^{2}\xi^{1-2s}v =0 &\mbox{ in }\mathcal{S}_{T}\!:=\!(0,T)^{N}\!\times\!(0,\infty)  \\
v_{| {\{x_{i}=0\}}}= v_{| {\{x_{i}=T\}}} & \mbox{ on } \partial_{L}\mathcal{S}_{T}\!:=\!\partial (0,T)^{N} \!\times \![0,\infty) \\
\frac{\partial v}{\partial \nu^{1-2s}}=\kappa_{s} [m^{2s}v+f(x,v)]   &\mbox{ on }\partial^{0}\mathcal{S}_{T}\!:=\!(0,T)^{N} \!\times\! \{0\}
\end{array},
\right.
\end{equation}
where 
$$
\frac{\partial v}{\partial \nu^{1-2s}}:=-\lim_{\xi \rightarrow 0} \xi^{1-2s} \frac{\partial v}{\partial \xi}(x,\xi)
$$
is the conormal exterior derivative of $v$.\\
Since (\ref{R}) has a variational nature, its solutions can be found as critical points of the functional 
 $$
 \mathcal{J}_{m}(v)=\frac{1}{2} ||v||_{\mathbb{X}_{m,T}^{s}}^{2}-\frac{m^{2s}\kappa_{s}}{2}|v(\cdot,0)|_{L^{2}(0,T)^{N}}^{2} -\kappa_{s}\int_{\partial^{0}\mathcal{S}_{T}} F(x,v) \,dx
 $$
defined on the space $\mathbb{X}_{m,T}^{s}$, which is the closure of the set of smooth and $T$-periodic (in $x$) functions in $\R^{N+1}_{+}$ with respect to the norm
$$
||v||_{\mathbb{X}_{m,T}^{s}}:=\sqrt{\iint_{\mathcal{S}_{T}} \xi^{1-2s} (|\nabla v|^{2}+m^{2s} v^{2}) dx d\xi}\,.
$$ 
Under the assumptions $(f1)-(f6)$ we are able to prove that for any $m>0$ fixed, $\mathcal{J}_{m}$ has a Linking geometry. 
We recall that in \cite{A2, A3} we proved that when $f$ satisfies $(f1)-(f4)$, $tf(x,t)\geq 0$ in $\R^{N+1}$ and the Ambrosetti-Rabinowitz condition \cite{AR}, i.e.
\begin{equation}
\exists \, \mu >2, \, R>0 : 0<\mu F(x,t) \leq f(x,t)t, \, \forall \, |t|\geq R,  \tag{AR}
\end{equation}
then we can obtain a nontrivial solution to (\ref{R}) by applying the standard Linking Theorem \cite{Rab, Struwe, Willem}.
This is due to the fact that (AR) guarantees the boundedness of Palais-Smale sequences for the functional associated with the problem under consideration.
However, although (AR) is a quite natural condition when we have to deal with superlinear elliptic problems, it is somewhat restrictive. 
In fact, by a direct integration of (AR), we can deduce the existence of $A,B>0$ such that
$$
F(x,t)\geq A|t|^{\mu}-B \mbox{ for any } t\in \R,
$$
which implies, being $\mu>2$, the condition $(f5)$.
If we consider the function $f(x,t)=t\log(1+|t|)$, then it is easy to prove that $f$ verifies $(f5)$ but it does not verify (AR). 
This means that that there are functions which are superlinear at infinity, but do not satisfy (AR).
For this reason, in several works concerning superlinear problems of the type
\begin{equation}\label{Ellpb}
\left\{
\begin{array}{ll}
\mathcal{L} u=f(x,u) &\mbox{ in }\Omega  \\
u=0 &\mbox{ on }  \partial \Omega
\end{array},
\right.
\end{equation}
where $\mathcal{L}$ is a second order elliptic operator and $\Omega\subset \R^{N}$ is a smooth bounded domain, some authors tried to drop the condition (AR); see for instance \cite{CM,J,Liu, MS, SZ} and references therein.\\
In this paper, we claim to study the non-local counterpart of (\ref{Ellpb}) with $\mathcal{L}=-\Delta+m^{2}$, in periodic setting, without assuming (AR). Our basic assumptions on the nonlinearity $f$ are $(f1)-(f6)$. We recall that the hypothesis $(f6)$ was introduced for the first time by Jeanjean in \cite{J} to show the existence of a bounded Palais-Smale sequence for functionals having a Mountain Pass structure
$$
I(u)=\frac{1}{2}||u||^{2}_{H^{1}(\R^{N})}-\int_{\R^{N}} F(x,u) dx
$$
and $f(x,t)=\partial_{t}F(x,t)$ does not satisfy (AR).\\
\noindent 
Here, in order to study the problem (\ref{R}), we invoke a variant of the Linking Theorem proved by Li and Wang in \cite{LW}, in which the Palais-Smale condition is replaced by the Cerami condition \cite{Cerami}; namely any  Cerami sequence $\{v_{j}\}$ in $\mathbb{X}_{m,T}^{s}$ at the level $\alpha \in \R$, that is  such that
$$
\mathcal{J}_{m}(v_{j})\rightarrow \alpha \, \mbox{ and }\, ||\mathcal{J}'_{m}(v_{j})||_{(\mathbb{X}_{m,T}^{s})^{*}}(1+||v_{j}||_{\mathbb{X}_{m,T}^{s}})\rightarrow 0 \, \mbox{ as } \, j\rightarrow \infty,
$$
admits a convergent subsequence.
At this point, to get the existence of a weak solution to (\ref{P}), it will be sufficient to show that every Cerami sequence is bounded and that it possesses a convergent subsequence. 
This step will constitute the heart of the proof of Theorem \ref{thm1}.

\noindent
Finally, we will also consider the problem (\ref{P}) when $m=0$, that is  
\begin{equation}\label{P'}
\left\{
\begin{array}{ll}
(-\Delta_{x})^{s} u=f(x,u) &\mbox{ in } (0,T)^{N}   \\
u(x+Te_{i})=u(x)    &\mbox{ for all } x \in \R^{N}, \quad i=1, \dots, N
\end{array}.
\right.
\end{equation}
By passing to the limit in (\ref{R}) as $m\rightarrow 0$, we prove the existence of a nontrivial periodic solution to (\ref{P'}). A such procedure of limit is justified by the fact that we are able to estimate from below and from above, the critical levels $\alpha_{m}$ of the functionals $\mathcal{J}_{m}$ independently of $m$, provided that $m$ is sufficiently small.  
Therefore our second result can be stated as follows
\begin{thm}\label{thm2}
Let $f:\R^{N+1} \rightarrow \R$ satisfying the assumptions $(f1)-(f6)$. Then the problem (\ref{P'}) admits a 
nontrivial solution $u\in \mathbb{H}^{s}_{T}$. 
\end{thm}

\noindent
The paper is organized as follows: in Section $2$ we give some notations and preliminaries results about the involved fractional Sobolev spaces; in Section $3$ we consider the extension problem (\ref{R}) which localizes the non-local problem (\ref{P});
in Section $4$ we investigate the existence of solutions to $(\ref{R})$ via variational methods; finally we show that we can find a nontrivial solution to $(\ref{P'})$ by passing to the limit in $(\ref{R})$ as $m\rightarrow 0$.


\section{Notations and Preliminaries}
\noindent
In this section we introduce the notations and we collect some facts which will be used along the paper.\\
Let 
$$
\R^{N+1}_{+}=\{(x,\xi)\in \R^{N+1}: x\in \R^{N}, \xi>0 \}
$$
be the upper half-space in $\R^{N+1}$. 

We denote by $\mathcal{S}_{T}=(0,T)^{N}\times(0,\infty)$ the half-cylinder in $\R^{N+1}_{+}$ with basis $\partial^{0}\mathcal{S}_{T}=(0,T)^{N}\times \{0\}$
and lateral boundary $\partial_{L}\mathcal{S}_{T}=\partial (0,T)^{N}\times [0,+\infty)$.\\
With $||v||_{L^{r}(\mathcal{S}_{T})}$ we always denote the norm of $v\in L^{r}(\mathcal{S}_{T})$ and with $|u|_{L^{r}(0,T)^{N}}$ the $L^{r}(0,T)^{N}$ norm of $u \in L^{r}(0,T)^{N}$.\\
Let $s\in (0,1)$ and $m> 0$. Let $A\subset \R^{N}$ be a domain. 

We define $L^{2}(A\times \R_{+},\xi^{1-2s})$ as the set of measurable functions $v$ on $A\times \R_{+}$ such that
$$
\iint_{A\times \R_{+}} \xi^{1-2s} v^{2} dxd\xi<\infty.
$$
We say that  $v\in H^{1}_{m}(A\times \R_{+},\xi^{1-2s})$ if $v$ and its weak gradient $\nabla v$ belong to $L^{2}(A\times \R_{+},\xi^{1-2s})$.
The norm of $v$ in $H^{1}_{m}(A\times \R_{+},\xi^{1-2s})$ is given by
$$
\iint_{A\times \R_{+}} \xi^{1-2s} (|\nabla v|^{2}+m^{2}v^{2}) \,dxd\xi<\infty.
$$
It is clear that $H^{1}_{m}(A\times \R_{+},\xi^{1-2s})$ is a Hilbert space with the inner product
$$
\iint_{A\times \R_{+}} \xi^{1-2s} (\nabla v \nabla z+m^{2}v z) \,dxd\xi.
$$
When $m=1$, we set  $H^{1}(A\times \R_{+},\xi^{1-2s})=H^{1}_{1}(A\times \R_{+},\xi^{1-2s})$.\\
We denote by $\mathcal{C}^{\infty}_{T}(\R^{N})$ the space of functions 
$u\in \mathcal{C}^{\infty}(\R^{N})$ such that $u$ is $T$-periodic in each variable, that is
$$
u(x+Te_{i})=u(x) \mbox{ for all } x\in \R^{N}, i=1, \dots, N. 
$$
Let $u\in  \mathcal{C}^{\infty}_{T}(\R^{N})$. Then we know that 
$$
u(x)=\sum_{k\in \Z^{N}} c_{k} \frac{e^{\imath\omega k\cdot x}}{\sqrt{T^{N}}} \quad \mbox{ for all } x\in \R^{N},
$$
where 
$$
\omega=\frac{2\pi}{T} \quad \mbox{ and } \quad c_{k}=\frac{1}{\sqrt{T^{N}}} \int_{(0,T)^{N}} u(x)e^{-\imath k\omega \cdot x}dx \quad (k\in \Z^{N})
$$ 
are the Fourier coefficients of $u$. We define the fractional Sobolev space $\mathbb{H}^{m}_{T}$ as the closure of 
$\mathcal{C}^{\infty}_{T}(\R^{N})$ under the norm 
\begin{equation*}\label{h12norm}
|u|_{\mathbb{H}^{s}_{m,T}}:=\sqrt{ \sum_{k\in \Z^{N}} (\omega^{2}|k|^{2}+m^{2})^{s} \, |c_{k}|^{2}}. 
\end{equation*}
We will also use the notation
$$
[u]_{\mathbb{H}^{s}_{m,T}}=\sqrt{\sum_{k\in \Z^{N}} \omega^{2}|k|^{2s}}
$$
to denote the Gagliardo semi-norm of $u$.

When $m=1$, we set $\mathbb{H}^{s}_{T}=\mathbb{H}^{s}_{1,T}$ and $|\cdot |_{\mathbb{H}^{s}_{T}}=|\cdot|_{\mathbb{H}^{s}_{1,T}}$. Finally we introduce the functional space $\mathbb{X}^{s}_{m,T}$ defined as the completion of 
\begin{align*}
\mathcal{C}_{T}^{\infty}(\overline{\R^{N+1}_{+}})=\Bigl\{&v\in \mathcal{C}^{\infty}(\overline{\R^{N+1}_{+}}): v(x+Te_{i},\xi)=v(x,\xi) \\
&\mbox{ for every } (x,\xi)\in \overline{\R_{+}^{N+1}}, i=1, \dots, N \Bigr\}
\end{align*}
under the $H^{1}_{m}(\mathcal{S}_{T},\xi^{1-2s})$ norm 
\begin{equation*}
||v||_{\mathbb{X}^{s}_{m,T}}:=\sqrt{\iint_{\mathcal{S}_{T}} \xi^{1-2s} (|\nabla v|^{2}+m^{2}v^{2}) \, dxd\xi}.
\end{equation*} 
If $m=1$, we set $\mathbb{X}^{s}_{T}=\mathbb{X}^{s}_{1,T}$ and $||\cdot ||_{\mathbb{X}^{s}_{T}}=||\cdot||_{\mathbb{X}^{s}_{1,T}}$. \\
In order to lighten the notation, it is convenient to omit the indices $s$ and $T$ (which are fixed) appearing in the Sobolev spaces and norms just defined. From now on we will write $\mathbb{X}_{m}$, $\mathbb{X}$, $\mathbb{H}_{m}$, $\mathbb{H}$, $||\cdot||_{\mathbb{X}_{m}}$, $||\cdot||_{\mathbb{X}}$, $|\cdot|_{\mathbb{H}}$ and $|\cdot|_{\mathbb{H}_{m}}$. \\  
Now we recall that it is possible to define a trace operator from $\mathbb{X}_{m}$ to $\mathbb{H}_{m}$:
\begin{thm}\cite{A2, A3}\label{tracethm}
There exists a surjective linear operator $\textup{Tr} : \mathbb{X}_{m} \rightarrow \mathbb{H}_{m}$  such that:
\begin{compactenum}[(i)]
\item $\textup{Tr}(v)=v|_{\partial^{0} \mathcal{S}_{T}}$ for all $v\in \mathcal{C}_{T}^{\infty}(\overline{\R^{N+1}_{+}}) \cap \mathbb{X}_{m}$;
\item $\textup{Tr}$ is bounded and 
\begin{equation}\label{tracein}
\kappa_{s} |\textup{Tr}(v)|^{2}_{\mathbb{H}_{m}}\leq ||v||^{2}_{\mathbb{X}_{m}} \mbox{ for every } v\in \mathbb{X}_{m}.
\end{equation}
In particular, equality holds in (\ref{tracein}) for some $v\in \mathbb{X}_{m}$ if and only if $v$ weakly solves 
$$ 
-\dive(\xi^{1-2s} \nabla v)+m^{2}\xi^{1-2s}v =0 \, \mbox{ in } \, \mathcal{S}_{T}.
$$
\end{compactenum}
\end{thm} 
Finally we have the following embeddings:
\begin{thm}\cite{A2, A3}\label{compacttracethm}
Let $N> 2s$. Then  $\textup{Tr}(\mathbb{X}_{m})$ is continuously embedded in $L^{q}(0,T)^{N}$ for any  $1\leq q \leq 2^{\s}_{s}$.  Moreover,  $\textup{Tr}(\mathbb{X}_{m})$ is compactly embedded in $L^{q}(0,T)^{N}$  for any  $1\leq q < 2^{\s}_{s}$. 
\end{thm}


\section{Extension Method}
\noindent
As mentioned in the introduction, crucial to our results is that $(-\Delta_{x}+m^{2})^{s}$ is a nonlocal operator which can be realized through a local problem in $\mathcal{S}_{T}$. This result can be stated more precisely as follows:
\begin{thm}\cite{A2, A3}
Let $u\in \mathbb{H}_{m}$. Then there exists a unique $v\in \mathbb{X}_{m}$ such that
\begin{equation}\label{extPu}
\left\{
\begin{array}{ll}
-\dive(\xi^{1-2s} \nabla v)+m^{2}\xi^{1-2s}v =0 &\mbox{ in }\mathcal{S}_{T}  \\
v_{| {\{x_{i}=0\}}}= v_{| {\{x_{i}=T\}}} & \mbox{ on } \partial_{L}\mathcal{S}_{T} \\
v(\cdot,0)=u  &\mbox{ on } \partial^{0}\mathcal{S}_{T}
\end{array}
\right.
\end{equation}
and
\begin{align}\label{conormal}
-\lim_{\xi \rightarrow 0} \xi^{1-2s}\frac{\partial v}{\partial \xi}(x,\xi)=\kappa_{s} (-\Delta_{x}+m^{2})^{s}u(x) \mbox{ in } \mathbb{H}^{*}_{m},
\end{align}
where  $\mathbb{H}^{*}_{m}$ is the dual of $\mathbb{H}_{m}$. We call $v$ the extension of $u$.
\end{thm}

\begin{remark}
In \cite{A2, A3} we proved that if $u=\sum_{k\in \Z^{N}} c_{k} \frac{e^{\imath \omega k\cdot x}}{{\sqrt{T^{N}}}}\in \mathbb{H}_{m}$, then  its extension is given by
\begin{align*}
v(x,\xi)=\sum_{k\in \Z^{N}} c_{k} \theta_{k}(\xi) \frac{e^{\imath \omega k\cdot x}}{{\sqrt{T^{N}}}}\in \mathbb{X}_{m}
\end{align*}  
where $\displaystyle{\theta_{k}(\xi)=  \theta(\sqrt{\omega^{2} |k|^{2}+ m^{2}} \xi)}$ and $\theta \in H^{1}(\R_{+},\xi^{1-2s})$ solves the following ODE
\begin{equation} \label{ccv}
\left\{
\begin{array}{cc}
&\theta^{''}+\frac{1-2s}{\xi}\theta^{'}-\theta=0 \mbox{ in } \R_{+}  \\
&\theta(0)=1 \mbox{ and } \theta(\infty)=0
\end{array}.
\right.
\end{equation}
It is known (see \cite{Erd}) that $\displaystyle{\theta(\xi)=\frac{2}{\Gamma(s)} \Bigl(\frac{\xi}{2}\Bigr)^{s} K_{s}(\xi)}$ where $K_{s}$ is the Bessel function of second kind with order $s$, and
being $\displaystyle{K'_{s}=\frac{s}{y}K_{s}-K_{s-1}}$, we can see that
$$
\kappa_{s}:= \int_{0}^{\infty} \xi^{1-2s} (|\theta'(\xi)|^{2}+|\theta(\xi)|^{2}) d\xi=-\lim_{\xi \rightarrow 0} \xi^{1-2s}\theta'(\xi)=2^{1-2s} \frac{\Gamma(1-s)}{\Gamma(s)}.
$$
By using this fact we can deduce that $||v||_{\mathbb{X}_{m}}=\sqrt{\kappa_{s}}|u|_{\mathbb{H}_{m}}$.
\end{remark}

\noindent
Now we take advantage of the previous result to reformulate  nonlocal problems with periodic boundary conditions, in a local way.\\
Let $g \in \mathbb{H}^{*}_{m}$ and consider the following two problems:
\begin{equation}\label{P*}
\left\{
\begin{array}{ll}
(-\Delta_{x}+m^{2})^{s}u=g &\mbox{ in } (0,T)^{N}  \\
u(x+Te_{i})=u(x) & \mbox{ for all } x\in \R^{N}
\end{array}
\right.
\end{equation}
and
\begin{equation}\label{P**}
\left\{
\begin{array}{ll}
-\dive(\xi^{1-2s} \nabla v)+m^{2}\xi^{1-2s}v =0 &\mbox{ in }\mathcal{S}_{T}  \\
v_{| {\{x_{i}=0\}}}= v_{| {\{x_{i}=T\}}} & \mbox{ on } \partial_{L}\mathcal{S}_{T} \\
\frac{\partial v}{\partial \nu^{1-2s}}=g(x)  &\mbox{ on } \partial^{0}\mathcal{S}_{T}
\end{array}.
\right.
\end{equation}

\begin{defn}
We say that $v \in \mathbb{X}_{m}$ is a weak solution to (\ref{P**})
if for every $\phi \in \mathbb{X}_{m}$ it holds
$$
\iint_{\mathcal{S}_{T}} \xi^{1-2s} (\nabla v \nabla \phi + m^{2} v \phi  ) \, dxd\xi=\kappa_{s}\langle g, \textup{Tr}(\phi)\rangle_{\mathbb{H}^{*}_{m}, \mathbb{H}_{m}}. 
$$
\end{defn}

\begin{defn}
We say that $u\in \mathbb{H}_{m}$ is a weak solution to (\ref{P*}) if $u=\textup{Tr}(v)$
and $v$ is a weak solution to (\ref{P**}).
\end{defn}
\begin{remark}
Later, with abuse of notation, we will denote by $v(\cdot,0)$ the trace $\textup{Tr}(v)$ of a function $v\in \mathbb{X}_{m}$.
\end{remark}

\section{Periodic solutions in the cylinder $\mathcal{S}_{T}$}
\noindent
In this section we prove the existence of a nontrivial solution to $(\ref{P})$. As explained in the previous section, we know that the study of $(\ref{P})$ is equivalent to investigate the existence of weak solutions to
\begin{equation}\label{R1}
\left\{
\begin{array}{ll}
-\dive(\xi^{1-2s} \nabla v)+m^{2}\xi^{1-2s}v =0 &\mbox{ in }\mathcal{S}_{T}\!:=\!(0,T)^{N}\! \times\! (0,\infty)  \\
v_{| {\{x_{i}=0\}}}= v_{| {\{x_{i}=T\}}} & \mbox{ on } \partial_{L}\mathcal{S}_{T}\!:=\!\partial (0,T)^{N} \!\times\! [0,\infty) \\
\frac{\partial v}{\partial \nu^{1-2s}}=\kappa_{s} [m^{2s}v+f(x,v)]   &\mbox{ on }\partial^{0}\mathcal{S}_{T}\!:=\!(0,T)^{N} \!\times\! \{0\}
\end{array}.
\right.
\end{equation}
For simplicity, we will assume that $\kappa_{s}=1$.

We introduce the following functional on $\mathbb{X}_{m}$
$$
\mathcal{J}_{m}(v)=\frac{1}{2}||v||^{2}_{\mathbb{X}_{m}} -\frac{ m^{2s}}{2}|v(\cdot,0)|_{L^{2}(0,T)^{N}}^{2}- \int_{\partial^{0}\mathcal{S}_{T}} F(x,v) dx.
$$
By conditions $(f2)$-$(f4)$ we know that for any $\varepsilon>0$ there exists $C_{\varepsilon}>0$ such that
\begin{equation}\label{f}
|f(x,t)|\leq  2\varepsilon|t|+(p+1)C_{\varepsilon}|t|^{p} \quad \forall t\in{\R}, \, \forall x\in [0,T]^{N}
\end{equation}
and
\begin{equation}\label{F}
|F(x,t)|\leq  \varepsilon|t|^{2}+C_{\varepsilon}|t|^{p+1} \quad \forall t\in {\R}, \, \forall x\in [0,T]^{N}.
\end{equation}
Then, by Theorem \ref{compacttracethm}, follows that $\mathcal{J}_{m}$ is well defined on $\mathbb{X}_{m}$ and $\mathcal{J}_{m}\in \mathcal{C}^{1}(\mathbb{X}_{m},\R)$. 
By using Theorem \ref{tracethm}, we notice that the quadratic part of $\mathcal{J}_{m}$ is nonnegative, that is
\begin{equation}\label{partequad}
||v||^{2}_{\mathbb{X}_{m}} -m^{2s}|v(\cdot,0)|_{L^{2}(0,T)^{N}}^{2}\geq 0, 
\end{equation}
and the equality holds in (\ref{partequad}) if and only if 
\begin{equation}\label{eqpartequad}
v(x,\xi)=c\theta(m\xi) \mbox{ for some } c\in \R 
\end{equation}
(see Theorem $7$ in \cite{A3}).\\
As observed in \cite{A2, A3}, $\mathbb{X}_{m}$ can be splitted as
$$
\mathbb{X}_{m}=<\theta(m\xi)> \oplus \left\{ v\in \mathbb{X}_{m}: \int_{(0,T)^{N}} v(x,0) dx=0\right \}=:\mathbb{Y}_{m} \oplus \mathbb{Z}_{m}
$$
where $\dim \mathbb{Y}_{m}<\infty$ and $\mathbb{Z}_{m}$ is the orthogonal complement of $\mathbb{Y}_{m}$ with respect to the inner product in $\mathbb{X}_{m}$.\\
In order to find critical points of $\mathcal{J}_{m}$ we will make use of a suitable version of the Linking Theorem due to Li and Wang \cite{LW}. \\
Firstly we recall the following definitions.
\begin{defn}
Let $(X, ||\cdot||_{X})$ be a real Banach space with its dual space $(X', ||\cdot||_{X'})$, $J\in C^{1}(X,\R)$ and $c\in \R$. 
We say that $\{v_{n} \}\subset X$ is a Cerami sequence for $J$ at the level $c$ if 
$$
J(v_{n})\rightarrow c \mbox{ and } (1+||v_{n}||_{X})||J'(v_{n})||_{X'} \rightarrow 0.
$$
\end{defn}
\begin{defn}
We say that $J$ satisfies the $(C)_{c}$ condition if any Cerami sequence at the level $c$ has a strongly convergent subsequence.
\end{defn}

\noindent
Now we are ready to state the following
\begin{thm} \cite{LW}\label{Linking Thm}
Let $(X, ||\cdot||_{X})$ be a real Banach space with $X=Y\bigoplus Z$, where $Y$ is finite dimensional. 
Let $\rho>r>0$ and let $z\in Z$ be a fixed element such that $||z||=r$. Define
\begin{align*}
M&:=\{v=y+\lambda z: y\in Y, ||v||_{X}\leq \rho, \lambda\geq 0\}, \\
M_{0}&:=\{v=y+\lambda z: y\in Y, ||v||_{X}= \rho, \lambda\geq 0 \mbox{ or } ||v||_{X}\leq \rho, \lambda=0 \}, \\
N_{r}&:=\{ v\in Z: ||v||_{X}=r \}.
\end{align*}
Let $J\in \mathcal{C}^{1}(X,\R)$ be such that
$$
b:=\inf_{N_{r}} J>a:=\max_{M_{0}} J.
$$
Then $\alpha \geq b$ and there exists a $(C)_{c}$ sequence of $J$ where
\begin{equation}
\alpha:=\inf_{\gamma \in \Gamma} \max_{v\in M} J(\gamma(v)) \mbox{ and } \Gamma:=\{ \gamma \in C(M,X): \gamma=Id \mbox{ on } M_{0} \}.
\end{equation}
In particular, if $J$ satisfies the $(C)_{\alpha}$ condition, then $\alpha$ is a critical value of $J$.
\end{thm}
\noindent
Then, we are going to verify that $\mathcal{J}_{m}$ satisfies the assumptions of the above Theorem \ref{Linking Thm}. We begin proving a series of lemmas, which ensure us that $\mathcal{J}_{m}$ possesses a Linking geometry.
\begin{lem}\label{lemma3}
$\mathcal{J}_{m}\leq 0 \mbox{ on } \mathbb{Y}_{m}$.
\end{lem}
\begin{proof}
Firstly we show that $F(x,t) \geq 0$ for any $x\in \R^{N}$ and $t\in \R$.
By $(f6)$ we deduce that $f(x,t)-2F(x,t)\geq 0$ for any $t\geq 0$.
Let $t>0$. For $x\in \R^{N}$ we have
$$
\frac{\partial}{\partial t}\Bigl(\frac{F(x,t)}{t^{2}}\Bigr)=\frac{t^{2}f(x,t)-2tF(x,t)}{t^{4}}\geq 0.
$$
By $(f2)$ we know that 
$$
\lim_{t\rightarrow 0} \frac{F(x,t)}{t^{2}}=0
$$ 
so we deduce that $F(x,t) \geq 0$ for $t\geq 0$.  Arguing similarly for the case $t\leq 0$, eventually we obtain that $F\geq 0$ in $\R^{N}\times \R$.
As a consequence, recalling that $\displaystyle{||v||^{2}_{\mathbb{X}_{m}}=m^{2s}|v(\cdot,0)|_{L^{2}(0,T)^{N}}^{2}}$ for any $v\in \mathbb{Y}_{m}$ by (\ref{eqpartequad}), we can see that
$$
\mathcal{J}_{m}(v)=-\int_{(0,T)^{N}} F(x,v) dx \leq 0.
$$
\end{proof}

\begin{lem}\label{lemma4}
There exists $r_{m}>0$ such that
$$
b_{m}:=\inf \{\mathcal{J}_{m}(v): v\in \mathbb{Z}_{m}, ||v||_{\mathbb{X}_{m}}=r_{m}\}>0.
$$
\end{lem}
\begin{proof}
By using Lemma $3$ in \cite{A3} we know that there is a constant $C_{m}>0$ such that
\begin{align}\label{eqnorm}
||v||_{\mathbb{X}_{m}}^{2}-m^{2s}|v(\cdot,0)|_{L^{2}(0,T)^{N}}^{2}\geq C_{m}||v||_{\mathbb{X}_{m}}^{2}
\end{align}
for any $v\in \mathbb{Z}_{m}$.
Then, taking into account $(\ref{F})$, $(\ref{eqnorm})$ and Theorem $4$ we have
\begin{align*}
\mathcal{J}_{m}(v)&\geq C_{m}||v||_{\mathbb{X}_{m}}^{2}-\varepsilon|v(\cdot,0)|_{L^{2}(0,T)^{N}}^{2}-C_{\varepsilon} |v(\cdot,0)|_{L^{p+1}(0,T)^{N}}^{p+1} \\
&\geq \Bigl(C_{m}-\frac{\varepsilon}{m}\Bigr)||v||_{\mathbb{X}_{m}}^{2}-C_{\varepsilon,m}||v||_{\mathbb{X}_{m}}^{p+1}
\end{align*}
for any $v\in \mathbb{Z}_{m}$.
Choosing $\varepsilon \in (0,mC_{m})$, we can find $r_{m}>0$ such that
$$
b_{m}:=\inf \{\mathcal{J}_{m}(v):  v\in \mathbb{Z}_{m} \mbox{ and } ||v||_{\mathbb{X}_{m}}=r_{m}  \}>0.
$$ 
\end{proof}

\begin{lem}\label{lemma5}
There exists $\rho_{m}>r_{m}$ and $z\in \mathbb{Z}_{m}$ with $||z||_{\mathbb{X}_{m}}=r_{m}$ such that,
denoted by
$$
M^{m}=\{v=y+\lambda z: y\in \mathbb{Y}_{m}, ||v||_{\mathbb{X}_{m}}\leq \rho_{m}, \lambda\geq 0 \} 
$$
and
$$
M^{m}_{0}=\{v=y+\lambda z: y\in \mathbb{Y}_{m}, ||v||_{\mathbb{X}_{m}}= \rho_{m}, \lambda\geq 0 \mbox{ or } ||v||_{\mathbb{X}_{m}}\leq \rho_{m}, \lambda=0 \}, 
$$
we have
$$
\max_{M^{m}_{0}} \mathcal{J}_{m}(v)\leq 0. 
$$
\end{lem}
\begin{proof}
By Lemma \ref{lemma3} we know that $\mathcal{J}_{m}(v)\leq 0$ on $\mathbb{Y}_{m}$.\\
Let us define
\begin{equation}\label{defw}
w=\prod_{i=1}^{N} \sin(\omega x_{i}) \frac{1}{\xi+1}.
\end{equation}
We observe that $w\in \mathbb{Z}_{m}$ (since $\int_{0}^{T} \sin(\omega x) dx=0$) and 
\begin{align}\label{w0}
||w||^{2}_{\mathbb{X}_{m}}&=N\Bigl(\prod_{i=1}^{N-1} \int_{0}^{T} \sin^{2}(\omega x_{i}) dx_{i}\Bigr)
\omega^{2} \Bigl( \int_{0}^{T} \cos^{2}(\omega x) dx\Bigr)\Bigl(\int_{0}^{\infty} \xi^{1-2s} \frac{d\xi}{(\xi+1)^{2}} \Bigr) \nonumber \\
&+\Bigl(\prod_{i=1}^{N} \int_{0}^{T} \sin^{2}(\omega x_{i}) dx_{i} \Bigr) \Bigl(\int_{0}^{\infty} \xi^{1-2s} \frac{d\xi}{(\xi+1)^{4}} \Bigr) \nonumber\\
&+m^{2}\Bigl(\prod_{i=1}^{N} \int_{0}^{T} \sin^{2}(\omega x_{i}) dx_{i} \Bigr) \Bigl(\int_{0}^{\infty} \xi^{1-2s} \frac{d\xi}{(\xi+1)^{2}} \Bigr).
\end{align}
Since
\begin{align*}
|w(\cdot,0)|_{L^{2}(0,T)^{N}}^{2}=\prod_{i=1}^{N} \int_{0}^{T} \sin^{2}(\omega x_{i}) dx_{i}
\end{align*}
and 
$$
\int_{0}^{T} \sin^{2}(\omega x) dx=\frac{T}{2}=\int_{0}^{T} \cos^{2}(\omega x) dx,
$$
by (\ref{w0}) follows that there exist $C_{1}, C_{2}, C_{3}>0$ (independent of $m$) such that 
\begin{equation}\label{einstein}
C_{1}|w(\cdot,0)|_{L^{2}(0,T)^{N}}^{2}\leq ||w||^{2}_{\mathbb{X}_{m}}\leq (C_{2}+m^{2}C_{3})|w(\cdot,0)|_{L^{2}(0,T)^{N}}^{2}.
\end{equation}

\noindent
Now, let $\displaystyle{z=\frac{r_{m}w}{||w||_{\mathbb{X}_{m}}}}$.
It is clear that $z\in \mathbb{Z}_{m}$ and $||z||_{\mathbb{X}_{m}}=r_{m}$. \\
Moreover, by (\ref{einstein}) we obtain
\begin{equation}\label{einstein1}
\frac{||z||^{2}_{\mathbb{X}_{m}}}{C_{2}+m^{2}C_{3}}=\frac{r^{2}_{m}}{C_{2}+m^{2}C_{3}}\leq |z(\cdot,0)|_{L^{2}(0,T)^{N}}^{2}\leq \frac{r^{2}_{m}}{C_{1}}=\frac{||z||^{2}_{\mathbb{X}_{m}}}{C_{1}}.
\end{equation}

\noindent
Take $\displaystyle{v=y+\lambda z\in \mathbb{Y}_{m}\oplus \R_{+} z}$. 
We recall that if $y\in \mathbb{Y}_{m}$ then $y(x,\xi)=d_{m}\theta(m\xi)$, $d_{m}\in \R$ and $||y||_{\mathbb{X}_{m}}^{2}=m^{2s}|y(\cdot,0)|_{L^{2}(0,T)^{N}}^{2}$. \\
Then, by using  (\ref{einstein1}), we have 
\begin{align}\label{einstein2}
||v||_{\mathbb{X}_{m}}^{2}&=||y||_{\mathbb{X}_{m}}^{2}+\lambda^{2}||z||_{\mathbb{X}_{m}}^{2} \nonumber \\
&=m^{2s}|y(\cdot,0)|_{L^{2}(0,T)^{N}}^{2}+\lambda^{2}||z||_{\mathbb{X}_{m}}^{2}  \nonumber \\
&\leq \max\{m^{2s},1\} [|y(\cdot,0)|_{L^{2}(0,T)^{N}}^{2}+\lambda^{2}||z||_{\mathbb{X}_{m}}^{2}]  \nonumber \\
&\leq \max\{m^{2s},1\} \{|y(\cdot,0)|_{L^{2}(0,T)^{N}}^{2}+\lambda^{2}[C_{2}+m^{2}C_{3}]|z(\cdot,0)|_{L^{2}(0,T)^{N}}^{2}\} \nonumber \\
&\leq \max\{m^{2s},1\} \max\{1,C_{2}+m^{2}C_{3} \} \{|y(\cdot,0)|_{L^{2}(0,T)^{N}}^{2}+\lambda^{2}|z(\cdot,0)|_{L^{2}(0,T)^{N}}^{2}\} \nonumber \\
&=:\bar{C}(m,s) |v(\cdot,0)|_{L^{2}(0,T)^{N}}^{2}.
\end{align}

\noindent
By using $(f2)$ and $(f5)$, we know that for any  $A>0$ there exists $B_{A}>0$  such that  
\begin{equation}\label{F*}
F(x,t)\geq A|t|^{2}-B_{A} \quad \forall t\in{\R}, \, \forall x\in [0, T]^{N}.
\end{equation}
Fix $A>\frac{\bar{C}(m,s)}{2}$.
Taking into account (\ref{einstein2}) and (\ref{F*}), we get for any $v=y+\lambda z\in \mathbb{Y}_{m}\oplus \R_{+} z$
\begin{align}
\mathcal{J}_{m}(v)&=\frac{1}{2}||v||_{\mathbb{X}_{m}}^{2}-\frac{m^{2s}}{2}|v(\cdot,0)|_{L^{2}(0,T)^{N}}^{2}-\int_{\partial^{0}\mathcal{S}_{T}} F(x,v) dx \nonumber \\
&\leq \frac{1}{2}||v||_{\mathbb{X}_{m}}^{2}-A|v(\cdot,0)|_{L^{2}(0,T)^{N}}^{2}+B_{A}T^{N} \nonumber \\
&\leq \Bigl[\frac{1}{2}-\frac{A}{\bar{C}(m,s)}\Bigr]||v||_{\mathbb{X}_{m}}^{2}+B_{A}T^{N}  \nonumber \rightarrow -\infty
\end{align}
as $||v||_{\mathbb{X}_{m}}=\rho_{m}\rightarrow \infty$.\\
Take $\rho_{m}$ large enough, $r_{m}$ small enough with $\rho_{m}>1>r_{m}>0$. \\
Then
$$
\max_{M^{m}_{0}} \mathcal{J}_{m}(v)\leq 0
$$
and
$$
b_{m}:=\inf_{N^{m}_{r_{m}}} \mathcal{J}_{m}>a_{m}:=\max_{M^{m}_{0}} \mathcal{J}_{m}.
$$
\end{proof}

\noindent
Finally we show that $\mathcal{J}_{m}$ satisfies the Cerami condition:

\begin{lem}\label{lemma6}
Let $c \in \R$. Let $\{v_{j}\}\subset \mathbb{X}_{m}$ be a sequence such that 
\begin{equation}\label{3.17}
\mathcal{J}_{m}(v_{j}) \rightarrow c 
\end{equation}
and
\begin{equation}\label{3.18}
(1+ ||v_{j}||_{\mathbb{X}_{m}}) ||\mathcal{J}_{m}'(v_{j})||_{\mathbb{X}^{*}_{m}} \rightarrow 0
\end{equation}
as $j\rightarrow \infty$.
Then there exists a subsequence $\{v_{j_{h}}\}\subset \{v_{j}\}$ and $v \in \mathbb{X}_{m}$ such that $v_{j_{h}} \rightarrow v$ in $\mathbb{X}_{m}$.
\end{lem}
\begin{proof}

We start proving that $\{v_{j}\}$ is bounded in $\mathbb{X}_{m}$. We proceed as in \cite{Liu}.
We argue by contradiction and assume that 
\begin{equation}\label{3.19}
||v_{j}||_{\mathbb{X}_{m}}\rightarrow \infty
\end{equation} 
as $j\rightarrow +\infty$.\\
Let us define
\begin{equation}\label{3.21}
z_{j}=\frac{v_{j}}{||v_{j}||_{\mathbb{X}_{m}}}.
\end{equation}
Then $||z_{j}||_{\mathbb{X}_{m}}=1$ and by using Theorem \ref{compacttracethm} we can assume, up to a subsequence,  that
\begin{align}\label{3.22}
z_{j} &\rightharpoonup z \mbox{ in } \mathbb{X}_{m}  \nonumber \\
z_{j}(\cdot,0) &\rightarrow z(\cdot,0) \mbox{ in } L^{q}(0,T)^{N} \mbox{ for any } q\in [2,2^{\s}_{s})\\
z_{j}(\cdot,0) &\rightarrow  z(\cdot,0) \mbox{ a.e in } (0,T)^{N} \nonumber
\end{align}
and there exists $h\in L^{p+1}(0,T)^{N}$ such that
\begin{equation}\label{3.23}
|z_{j}(x,0)|\leq h(x) \quad \mbox{ a.e. in } x\in (0,T)^{N}, \quad \mbox{ for all } j\in \N.
\end{equation}
Now we distinguish two cases.
Firstly we suppose that 
\begin{equation}\label{3.24}
z\equiv 0.
\end{equation}
As in \cite{J}, we can choose $\{t_{j}\}_{j\in \N}\subset [0,1]$ such that
\begin{equation}\label{3.25}
\mathcal{J}_{m}(t_{j}v_{j})=\max_{t\in [0,1]} \mathcal{J}_{m}(t v_{j}).
\end{equation}
Since $||v_{j}||_{\mathbb{X}_{m}}\rightarrow \infty$ we can take $r_{n}=2\sqrt{n}$
such that
\begin{equation}\label{3.26}
r_{n}||v_{j}||_{\mathbb{X}_{m}}^{-1}\in (0,1)
\end{equation}
provided $j$ is large enough.
By (\ref{3.22}) and the continuity of $F$, we can see 
\begin{equation}\label{3.27}
F(x,r_{n}z_{j}(x,0))\rightarrow F(x,r_{n}z(x,0)) \mbox{ a.e. } x\in (0,T)^{N}
\end{equation}
as $j\rightarrow \infty$ and $n\in \N$. 
In particular, integrating $(f4)$ and exploiting (\ref{3.23}) we get
\begin{align}\label{3.28}
|F(x,r_{n}z_{j}(x,0))|&\leq c_{1}|r_{n}z_{j}(x,0)|+c_{2}|r_{n}z_{j}(x,0)|^{p+1} \nonumber \\
&\leq c_{1}r_{n} h(x)+c_{2}r_{n}^{p}h(x)^{p+1}\in L^{1}(0,T)^{N},
\end{align}
a.e. $x\in (0,T)^{N}$ and $n, j\in \N$.
Then, taking into account (\ref{3.27}), (\ref{3.28}) and by using the Dominated Convergence Theorem we deduce that 
\begin{equation}\label{3.29}
F(x,r_{n}z_{j}(x,0))\rightarrow F(x,r_{n}z(x,0)) \mbox{ in } L^{1}(0,T)^{N}.
\end{equation}
Since $F(\cdot,0)=0$ and (\ref{3.24}) holds true, (\ref{3.29}) yields
\begin{equation}\label{3.30}
\int_{\partial^{0}\mathcal{S}_{T}}F(x,r_{n}z_{j}) dx\rightarrow 0 \mbox{ as } j\rightarrow \infty
\end{equation}
for any $n\in \N$.
Then (\ref{3.22}), (\ref{3.25}), (\ref{3.26}) and (\ref{3.30}) imply
$$
\mathcal{J}_{m}(t_{j}v_{j})\geq \mathcal{J}_{m}(r_{n}z_{j})\geq 2n-\int_{\partial^{0}\mathcal{S}_{T}} F(x,r_{n}z_{j}) dx\geq n
$$
provided $j$ is large enough and for any $n\in \N$. 
As a consequence
\begin{equation}\label{3.31}
\mathcal{J}_{m}(t_{j}v_{j})\rightarrow \infty \mbox{ as } j\rightarrow \infty.
\end{equation}
Since $\mathcal{J}_{m}(0)=0$ and $\mathcal{J}_{m}(v_{j})\rightarrow c$ we deduce that $t_{j}\in (0,1)$. Thus, by (\ref{3.25}) we have
\begin{equation}\label{3.32}
\langle{\mathcal{J}'_{m}(t_{j}v_{j}),t_{j}v_{j}}\rangle=t_{j} \frac{d}{dt}\Bigr|_{t=t_{j}} \mathcal{J}_{m}(t v_{j})=0.
\end{equation} 
Taking into account $(f6)$, (\ref{3.17}), (\ref{3.18}) and (\ref{3.32}) we get
\begin{align*}
\frac{2}{\gamma}\mathcal{J}_{m}(t_{j}v_{j})&=\frac{2}{\gamma}\Bigl(\mathcal{J}_{m}(t_{j}v_{j})- \frac{1}{2}\langle{\mathcal{J}'_{m}(t_{j}v_{j}),t_{j}v_{j}}\rangle \Bigr)\\
&=\frac{1}{\gamma} \int_{\partial^{0}\mathcal{S}_{T}}G(x,t_{j}v_{j}) dx \\
&\leq \int_{\partial^{0}\mathcal{S}_{T}}G(x,v_{j}) dx \\
&=2[\mathcal{J}_{m}(v_{j})-\frac{1}{2}\langle{\mathcal{J}'_{m}(v_{j}),v_{j}}\rangle] \rightarrow 2c \quad \mbox{ as } j\rightarrow \infty
\end{align*}
which contradicts (\ref{3.31}).\\
Secondly, we suppose that
\begin{equation}\label{zdiv0}
z\not\equiv 0.
\end{equation}
Thus the set $\Omega=\{x\in (0,T)^{N}: z(x,0)\neq 0\}$ has positive Lebesgue measure and by using (\ref{3.21}), (\ref{3.22}) and (\ref{zdiv0}) we get
\begin{equation}\label{3.34}
|v_{j}(x,0)|\rightarrow \infty \mbox{ a.e. } x\in \Omega \mbox{ as } j\rightarrow 0.
\end{equation}
By (\ref{3.17}), (\ref{3.19}) and $F\geq 0$ we can easily deduce that
\begin{equation}\begin{split}\label{3.35}
o(1)&=\frac{1}{2}-\frac{m^{2s}}{2}\frac{|v_{j}(\cdot,0)|_{L^{2}(0,T)^{N}}^{2}}{||v_{j}||^{2}_{\mathbb{X}_{m}}}-\int_{(0,T)^{N}}\frac{F(x,v_{j}(x,0))}{||v_{j}||^{2}_{\mathbb{X}_{m}}} dx  \\
&\leq \frac{1}{2}-\int_{\Omega}\frac{F(x,v_{j}(x,0))}{||v_{j}||^{2}_{\mathbb{X}_{m}}} dx \quad \mbox{ as } j\rightarrow \infty.
\end{split}
\end{equation}
Now, taking into account $(f5)$, (\ref{3.21}), (\ref{3.22}), (\ref{3.34}) and by using Fatou's Lemma we obtain
\begin{align}\label{3.36}
\int_{\Omega}\frac{F(x,v_{j}(x,0))}{||v_{j}||^{2}_{\mathbb{X}_{m}}} dx \rightarrow \infty
\mbox{ as } j\rightarrow \infty.
\end{align} 
Putting together (\ref{3.35}) and (\ref{3.36}) we get a contradiction. \\
Thus the sequence $\{v_{j}\}$ is bounded in $\mathbb{X}_{m} $.
By Theorem \ref{compacttracethm} we can assume, up to a subsequence,  that
\begin{equation}\begin{split}\label{3.7}
v_{j} &\rightharpoonup v \mbox{ in } \mathbb{X}_{m}  \\
v_{j}(\cdot,0) &\rightarrow v(\cdot,0) \mbox{ in } L^{p+1}(0,T)^{N} \\
v_{j}(\cdot,0) &\rightarrow  v(\cdot,0) \mbox{ a.e in } (0,T)^{N} 
\end{split}
\end{equation}
as $j\rightarrow \infty$ and there exists $h\in L^{p+1}(0,T)^{N}$ such that
\begin{equation}\label{3.10}
|v_{j}(x,0)|\leq h(x) \quad \mbox{ a.e. in } x\in (0,T)^{N}, \,\mbox{ for all } j\in \N.
\end{equation}
Taking into account $(f2)$, $(f4)$, (\ref{3.7}), (\ref{3.10}) and the Dominated Convergence Theorem  we get
\begin{equation}\label{1}
\int_{\partial^{0}\mathcal{S}_{T}} f(x,v_{j}) v_{j} dx \rightarrow \int_{\partial^{0}\mathcal{S}_{T}} f(x,v) v dx
\end{equation}
and
\begin{equation}\label{2}
\int_{\partial^{0}\mathcal{S}_{T}} f(x,v_{j}) v dx \rightarrow \int_{\partial^{0}\mathcal{S}_{T}} f(x,v) v dx
\end{equation}
as $j \rightarrow \infty$.\\
By using (\ref{3.18}) and the boundedness of $\{v_{j}\}_{j\in \N}$ in $\mathbb{X}_{m}$, we deduce that 
$\langle  \mathcal{J}_{m}'(v_{j}),v_{j} \rangle \rightarrow 0$, that is
\begin{equation}\label{wc1}
||v_{j}||_{\mathbb{X}_{m}}^{2}- m^{2s}|v_{j}(\cdot,0)|_{L^{2}(0,T)^{N}}^{2}-\int_{\partial^{0}\mathcal{S}_{T}} f(x,v_{j}) v_{j} dx \rightarrow 0
\end{equation}
as $j\rightarrow \infty$.
By $(\ref{3.7})$, (\ref{1}) and (\ref{wc1}) we have
\begin{equation}\label{wc2}
||v_{j}||_{\mathbb{X}_{m}}^{2} \rightarrow m^{2s}|v(\cdot,0)|_{L^{2}(0,T)^{N}}^{2}-\int_{\partial^{0}\mathcal{S}_{T}} f(x,v) v dx.
\end{equation}
Moreover, by (\ref{3.18}) and $v \in \mathbb{X}_{m}$, we obtain $\langle  \mathcal{J}_{m}'(v_{j}),v \rangle \rightarrow 0$ as $j\rightarrow \infty$, that is
\begin{equation}\label{wc3}
\langle v_{j}, v \rangle_{\mathbb{X}_{m}}- m^{2s} \langle v_{j},v \rangle_{L^{2}(0,T)^{N}}-\int_{\partial^{0}\mathcal{S}_{T}} f(x,v_{j}) v dx \rightarrow 0
\end{equation}
Taking into account (\ref{3.7}), (\ref{3.10}), (\ref{2}) and (\ref{wc3}) we get
\begin{equation}\label{wc4}
||v||_{\mathbb{X}_{m}}^{2}= m^{2s}|v(\cdot,0)|_{L^{2}(0,T)^{N}}^{2}-\int_{\partial^{0}\mathcal{S}_{T}} f(x,v) v dx.
\end{equation}
Thus, (\ref{wc2}) and (\ref{wc4}) imply that 
\begin{equation}\label{wc5}
||v_{j}||_{\mathbb{X}_{m}}^{2}\rightarrow ||v||_{\mathbb{X}_{m}}^{2} \mbox{ as } j\rightarrow \infty.
\end{equation}
Since $\mathbb{X}_{m}$ is a Hilbert space, we have
$$
||v_{j}-v||_{\mathbb{X}_{m}}^{2}=||v_{j}||_{\mathbb{X}_{m}}^{2}+||v||_{\mathbb{X}_{m}}^{2}-2 \langle v_{j},v \rangle_{\mathbb{X}_{m}}
$$
and using $v_{j} \rightharpoonup v$ in $\mathbb{X}_{m}$ and (\ref{wc5})
we can conclude that $v_{j} \rightarrow v$ in $\mathbb{X}_{m}$, as $j \rightarrow \infty$.

\end{proof}

\begin{proof} [Proof of Theorem \ref{thm1}]
Taking into account Lemma \ref{lemma3} - Lemma \ref{lemma6} we can see that $\mathcal{J}_{m}$ satisfies the assumptions of Theorem \ref{Linking Thm}.
Then, we deduce that for any fixed $m>0$, there exists a function $v_{m}\in \mathbb{X}_{m}$ such that
\begin{equation}\label{criticalvalue}
\mathcal{J}_{m}(v_{m})=\alpha_{m} \mbox{ and } \mathcal{J}'_{m}(v_{m})=0, 
\end{equation}
where 
\begin{equation}\label{am}
\alpha_{m}:=\inf_{\gamma \in \Gamma_{m}} \max_{v\in M^{m}} \mathcal{J}_{m}(\gamma(v)) 
\end{equation}
and 
\begin{equation*}
\Gamma_{m}:=\{ \gamma \in C(M^{m},\mathbb{X}_{m}): \gamma=Id \mbox{ on } M^{m}_{0} \}.
\end{equation*}
Hence $v_{m}$ is a nontrivial weak solution to (\ref{R1}), and by Theorem $9$ in \cite{A3} follows that $v_{m}(\cdot,0)\in C^{0,\alpha}([0,T]^{N})$ for some $\alpha \in (0,1)$.
\end{proof}

\section{Periodic solution in the case $m=0$}
\noindent
In this last section, we show that it is possible to find a nontrivial weak solution to (\ref{P'}). In Section $4$ we proved that for any $m>0$ there exists $v_{m}\in \mathbb{X}_{m}$ such that 
\begin{align}\label{zero}
\mathcal{J}_{m}(v_{m})= \alpha_{m} \quad \mbox{ and } \quad \mathcal{J}'_{m}(v_{m})=0,
\end{align} 
where $\alpha_{m}$ is defined as in (\ref{am}). 
Now, we will prove that we can estimate from below and from above the critical levels of the functional $\mathcal{J}_{m}$  independently of $m$, when $m$ is sufficiently small.  This allows us to pass to the limit in (\ref{R1}) as $m\rightarrow 0$ and to prove the existence of a nontrivial solution to 
\begin{equation}\label{R'}
\left\{
\begin{array}{ll}
-\dive(\xi^{1-2s} \nabla v) =0 &\mbox{ in }\mathcal{S}_{T}:=(0,T)^{N} \times (0,\infty)  \\
v_{| {\{x_{i}=0\}}}= v_{| {\{x_{i}=T\}}} & \mbox{ on } \partial_{L}\mathcal{S}_{T}:=\partial (0,T)^{N} \times [0,\infty) \\
\frac{\partial v}{\partial \nu^{1-2s}}= f(x,v)   &\mbox{ on }\partial^{0}\mathcal{S}_{T}:=(0,T)^{N} \times \{0\}
\end{array}.
\right.
\end{equation}

\noindent
Let us assume that $\displaystyle{0<m<m_{0}:=\frac{\omega^{2s}}{2}}$. \\
Firstly we prove that there exists $K_{1}>0$ independent of $m$, such that
\begin{equation}\label{cmlambda}
\alpha_{m}\geq K_{1} \quad \mbox{ for all }  0<m<m_{0}. 
\end{equation}
In order to achieve our aim, we will estimate the $L^{q}$-norm of the trace of $v$, with $q\in [2,2^{\s}_{s})$. 
Let $v\in \mathbb{Z}_{m}$ and we denote by $c_{k}$ its Fourier coefficients.
By $c_{0}=0$ and Theorem \ref{tracethm} follow that 
\begin{align}\label{intr*}
|v(\cdot,0)|^{2}_{L^{2}(0,T)^{N}} & = \sum_{|k|\geq 1}|c_{k}|^{2} 
 \leq \frac{1}{\omega^{2s}} \sum_{|k|\geq 1} (\omega^{2}|k|^{2}+m^{2})^{s}|c_{k}|^{2} \nonumber \\
& =\frac{1}{\omega^{2s}} |v(\cdot,0)|^{2}_{\mathbb{H}_{m}}
\leq \frac{1}{\omega^{2s}} ||v||^{2}_{\mathbb{X}_{m}}.
\end{align}
\noindent
Now, fix $2<q<2^{\s}_{s}$ and we denote by $q'$ its conjugate exponent. 
Taking into account $c_{0}=0$, by using H\"older inequality and Theorem \ref{tracethm} we get 
\begin{align*}
\Bigl(\sum_{|k|\geq 1} |c_{k}|^{q'}  \Bigr)^{\frac{1}{q'}}
&\leq   \Bigl[ |v(\cdot,0)|_{\mathbb{H}_{m}} \Bigl( \sum_{|k|\geq 1} ((\omega^{2}|k|^{2}+m^{2})^{s})^{-\frac{q'}{2-q'}} 
\Bigr)^{\frac{2-q'}{2q'}} \Bigr] \nonumber \\
&\leq \omega^{-s} \Bigl( \sum_{|k|\geq 1} |k|^ {-\frac{2sq'}{2-q'}}  \Bigr)^{\frac{2-q'}{2q'}}  
|v(\cdot,0)|_{\mathbb{H}_{m}} \nonumber \\
&\leq \frac{\omega^{-s} }{\sqrt{\kappa_{s}}}\Bigl( \sum_{|k|\geq 1} |k|^ {-\frac{2sq'}{2-q'}}  \Bigr)^{\frac{2-q'}{2q'}} || v ||_{\mathbb{X}_{m}}<\infty
\end{align*}
because of $1<\frac{2N}{N+2s}<q'<2$. \\
As a consequence, by using the Theorem of Hausdorff-Young-Riesz (see pages $101$-$102$ in \cite{Zy}) follows that
$$
|v(\cdot,0)|_{L^{q}(0,T)^{N}}\leq \Bigl(\frac{1}{\sqrt{T^{N}}}\Bigr)^{\frac{2}{q'}-1} \Bigl(\sum_{|k|\geq 1} |c_{k}|^{q'}  \Bigr)^{\frac{1}{q'}}
$$
and taking $q=p+1$ we have
\begin{align}\label{series}
|v(\cdot,0)|_{L^{p+1}(0,T)^{N}}&\leq \Bigl(\frac{1}{\sqrt{T^{N}}}\Bigr)^{\frac{2}{(p+1)'}-1}\Bigl(\sum_{|k|\geq 1} |c_{k}|^{(p+1)'}  \Bigr)^{\frac{1}{(p+1)'}}\nonumber \\
&\leq C'' ||v||_{\mathbb{X}_{m}}
\end{align}
for some $C'':=C''(T,N,s,p)>0$ independent of $m$.\\
Then, by using (\ref{F}) with $\displaystyle{0<\varepsilon<\frac{\omega^{2s}}{4}}$ and exploiting  (\ref{intr*}) and (\ref{series}), we can see that for every $\displaystyle{0<m<m_{0}=\frac{\omega^{2s}}{2}}$
\begin{align*}
\mathcal{J}_{m}(v)&=\frac{1}{2}\iint_{\mathcal{S}_{T}} y^{1-2s}(|\nabla v|^{2}+m^{2} v^{2}) \ dx dy-\frac{m^{2s}}{2}\int_{\partial^{0} \mathcal{S}_{T}} |v|^{2} dx -\int_{\partial^{0} \mathcal{S}_{T}} F(x,v) dx \\
&\geq \frac{1}{2} ||v||^{2}_{\mathbb{X}_{m}}-\Bigl(\frac{m}{2}+\varepsilon \Bigr) |v(\cdot,0)|_{L^{2}(0,T)^{N}}^{2} -C_{\varepsilon} |v(\cdot,0)|_{L^{p+1}(0,T)^{N}}^{p+1} \\
&\geq  \Bigl[\frac{1}{2}-\frac{1}{\omega^{2s}}\Bigl(\frac{m}{2}+\varepsilon \Bigr) \Bigr] ||v||^{2}_{\mathbb{X}_{m}} -C''_{\varepsilon} ||v||_{\mathbb{X}_{m}}^{p+1}    \\
& \geq \Bigl(\frac{1}{4}-\frac{\varepsilon}{\omega^{2s}}  \Bigr) ||v||^{2}_{\mathbb{X}_{m}} -C''_{\varepsilon}||v||_{\mathbb{X}_{m}}^{p+1} .
\end{align*}

\noindent
Set $\displaystyle{b:=\frac{1}{4}-\frac{\varepsilon}{\omega^{2s}}>0}$ and $\displaystyle{r:=\Bigl(\frac{b}{2C''_{b}}\Bigr)^{\frac{1}{p-1}}}$. Then, for every $v\in \mathbb{Z}_{m}$ such that $||v||_{\mathbb{X}_{m}}=r$
$$
\mathcal{J}_{m}(v) \geq  br^{2}-C''_{b}\, r^{p+1}= \frac{b}{2} \Bigl(\frac{b}{2C''_{b}} \Bigr)^{\frac{2}{p-1}}=:K_{1} 
$$
from which follows (\ref{cmlambda}).
Now, we show that there exists a positive constant $K_{2}>0$ independent of $m$ such that 
\begin{equation}\label{cmdelta}
\alpha_{m}\leq K_{2} \quad \mbox{ for all }  0<m<m_{0}. 
\end{equation}
\noindent
By using (\ref{einstein}) and $0<m<m_{0}$ we know that
$$
C_{1}|w(\cdot,0)|_{L^{2}(0,T)^{N}}^{2}\leq ||w||_{\mathbb{X}_{m}}^{2} \leq (C_{2}+m_{0}^{2}C_{3})|w(\cdot,0)|_{L^{2}(0,T)^{N}}^{2}
$$
where $w$ is the function defined in (\ref{defw}). \\
Set 
$$
z:=\frac{rw}{ ||w||_{\mathbb{X}_{m}}}.
$$
Recalling that $0<m<m_{0}$,
we can see that $(\ref{einstein2})$ in Section $4$ can be replaced by 
\begin{align}
||v||_{\mathbb{X}_{m}}^{2}&\leq \max\{m_{0}^{2s},1\} \max\{1,C_{2}+m_{0}^{2}C_{3} \} \{|y(\cdot,0)|_{L^{2}(0,T)^{N}}^{2}+\lambda^{2}|z(\cdot,0)|_{L^{2}(0,T)^{N}}^{2}\} \nonumber \\
&=:\bar{C}(m_{0},s) |v(\cdot,0)|_{L^{2}(0,T)^{N}}^{2} \nonumber
\end{align}
for any $v=y+\lambda z\in \mathbb{Y}_{m}\oplus \R_{+} z$.

Now, fix $A>\frac{\bar{C}(m_{0},s)}{2}$. By using (\ref{F*}) and $0<m<m_{0}$, we have for any $v=y+\lambda z\in \mathbb{Y}_{m}\oplus \R_{+} z$ 
\begin{align}
\mathcal{J}_{m}(v)&=\frac{1}{2}||v||_{\mathbb{X}_{m}}^{2}-\frac{m^{2s}}{2}|v(\cdot,0)|_{L^{2}(0,T)^{N}}^{2}-\int_{\partial^{0} \mathcal{S}_{T}} F(x,v) dx \nonumber\\
&\leq \frac{1}{2}||v||_{\mathbb{X}_{m}}^{2} -\frac{A}{\bar{C}(m_{0},s) }||v||_{\mathbb{X}_{m}}^{2}+B\, T^{N} \nonumber \\
&\leq B\, T^{N}=:K_{2}.
\end{align}
Therefore, taking into account  (\ref{zero}), (\ref{cmlambda}) and (\ref{cmdelta}) we deduce that 
\begin{align}\label{alpham}
K_{1} \leq \alpha_{m}\leq K_{2} \,  \mbox{ for every } \,  0<m<m_{0}.
\end{align}
Now, we show how to exploit this last information to pass to the limit in (\ref{R}) as $m\rightarrow 0$. 

Firstly we begin proving that for any $\delta>0$, holds the following inequality
\begin{align}\label{nash}
||v||_{L^{2}((0,T)^{N}\times (0,\delta),\xi^{1-2s})}^2 &\leq \frac{\delta^{2-2s}}{1-s} |v(\cdot,0)|_{L^{2}(0,T)^{N}}^{2}+\frac{\delta ^{2}}{2s} ||\partial_{\xi} v||_{L^{2}(\mathcal{S}_{T},\xi^{1-2s})}^{2}
\end{align}
for any $v\in \mathbb{X}_{m}$.

Fix $\delta>0$ and let $v\in C^{\infty}_{T}(\overline{\R^{N+1}_{+}})$ such that $||v||_{\mathbb{X}_{m}}<\infty$.
For any $x\in [0,T]^{N}$ and $\xi\in [0, \delta]$, we have
$$
v(x,\xi)=v(x,0)+\int_{0}^{\xi} \partial_{\xi} v(x,t) dt.
$$
By using $(a+b)^{2}\leq 2a^{2}+2b^{2}$ for all $a, b\geq 0$ we obtain
$$
|v(x,\xi)|^2 \leq 2  |v(x,0)|^{2}+2\Bigl(\int_{0}^{\xi}|\partial_{\xi} v(x,t)| dt\Bigr)^{2},
$$
and applying the H\"older inequality we deduce
\begin{equation}\label{vtii5}
|v(x,\xi)|^2 \leq 2 \Bigl[ |v(x,0)|^{2}+\Bigl(\int_{0}^{\xi} t^{1-2s}|\partial_{\xi} v(x,t)|^{2}dt\Bigr)\frac{\xi^{2s}}{2s}\,  \Bigr].
\end{equation}
Multiplying both members of (\ref{vtii5}) by $\xi^{1-2s}$ we have
\begin{equation}\label{vtii}
\xi^{1-2s}|v(x,\xi)|^2 \leq 2 \Bigl[ \xi^{1-2s}|v(x,0)|^{2}+\Bigl(\int_{0}^{\xi} t^{1-2s} |\partial_{\xi} v(x,t)|^{2}dt\Bigr)\frac{\xi}{2s} \Bigr].
\end{equation}
Integrating (\ref{vtii}) over $(0,T)^{N}\times (0,\delta)$ we have
\begin{align}\label{nash}
||v||_{L^{2}((0,T)^{N}\times (0,\delta),\xi^{1-2s})}^2 &\leq \frac{\delta^{2-2s}}{1-s} |v(\cdot,0)|_{L^{2}(0,T)^{N}}^{2}+\frac{\delta ^{2}}{2s} ||\partial_{\xi} v||_{L^{2}(\mathcal{S}_{T},\xi^{1-2s})}^{2}.
\end{align}
By density we get the desired result.\\
Secondly, we can observe that by the definition of the norm $||\cdot||_{\mathbb{X}_{m}}$ and by Theorem \ref{tracethm} follow that
\begin{equation}\label{1v}
||v_{m}||_{\mathbb{X}_{m}}\geq ||\nabla v_{m}||^{2}_{L^{2}(\mathcal{S}_{T},\xi^{1-2s})}
\end{equation}
and
\begin{align}\label{2v}
||v_{m}||_{\mathbb{X}_{m}}&\geq |v_{m}(\cdot,0)|_{\mathbb{H}_{m}}\geq [v_{m}(\cdot,0)]_{\mathbb{H}_{m}}.
\end{align}
Finally, we notice that (\ref{alpham}) implies that 
\begin{align*}
2\int_{\partial^{0}\mathcal{S}_{T}} F(x,v_{m}) dx&<2 K_{1}+m^{2s}|v_{m}(\cdot,0)|_{L^{2}(0,T)^{N}}^{2}+2\int_{\partial^{0}\mathcal{S}_{T}} F(x,v_{m}) dx\nonumber \\
&\leq ||v_{m}||^{2}_{\mathbb{X}_{m}}
\end{align*}
and by applying (\ref{F*}) with $A=1$, we can deduce that
\begin{align}\label{vti7}
2|v_{m}(\cdot,0)|_{L^{2}(0,T)^{N}}^{2} -2B_{1}T^{N}dx\leq ||v_{m}||^{2}_{\mathbb{X}_{m}}.
\end{align}
Then, taking into account  (\ref{nash}), (\ref{1v}), (\ref{2v}), (\ref{vti7}), and Theorem \ref{compacttracethm}, it is enough to prove
\begin{equation}\label{vmfinite}
\limsup_{m\rightarrow 0} ||v_{m}||_{\mathbb{X}_{m}}<\infty
\end{equation}
to deduce the  existence of a subsequence, that for simplicity we will denote again with  $\{v_{m}\}$, and a function $v$ such that 
\begin{align}\label{limitsforv}
&v\in L^{2}_{loc}(\mathcal{S}_{T},\xi^{1-2s}) \mbox{ and } \nabla v\in L^{2}(\mathcal{S}_{T},\xi^{1-2s});\nonumber \\
&v_{m}\rightharpoonup v \mbox{ in } L^{2}_{loc}(\mathcal{S}_{T},\xi^{1-2s}) \mbox{ as } m\rightarrow 0; \\
&\nabla v_{m}\rightharpoonup \nabla v \mbox{ in } L^{2}(\mathcal{S}_{T},\xi^{1-2s}) \mbox{ as } m\rightarrow 0; \nonumber  \\ 
&v_{m}(\cdot,0)\rightarrow v(\cdot,0) \mbox{ in } L^{q}(0,T)^{N} \mbox{ for any } q\in \Bigl[2, \frac{2N}{N-2s}\Bigr), \mbox{ as } m\rightarrow 0. \nonumber 
\end{align}
To show the validity of (\ref{vmfinite}), we proceed as in the first part of the proof of Lemma \ref{lemma6} in which we demonstrated the boundedness of Cerami sequences.
\noindent
We assume by contradiction that, up to a subsequence, 
\begin{equation}\label{vminfty}
||v_{m}||_{\mathbb{X}_{m}}\rightarrow \infty \mbox{ as } m\rightarrow 0.
\end{equation}
\noindent
We set
\begin{equation}\label{defwm}
w_{m}:=\frac{v_{m}}{||v_{m}||_{\mathbb{X}_{m}}}.
\end{equation}
Then $||w_{m}||_{\mathbb{X}_{m}}=1$ and by using (\ref{1v})-(\ref{2v}) results
\begin{equation}\label{nablavti}
||\nabla w_{m}||^{2}_{L^{2}(\mathcal{S}_{T},\xi^{1-2s})}\leq 1
\end{equation}
and
\begin{align}\label{compik}
1=||w_{m}||^{2}_{\mathbb{X}_{m}}\geq C(T,N,s)[w_{m}(\cdot,0)]_{\mathbb{H}_{m}}.
\end{align}
Moreover, by using (\ref{vti7}), (\ref{vminfty}) and (\ref{defwm}), we have
\begin{equation}\label{vti}
\limsup_{m\rightarrow 0}|w_{m}(\cdot,0)|_{L^{2}(0,T)^{N}}^{2}\leq \frac{1}{2}.
\end{equation}
Putting together (\ref{compik}) and (\ref{vti}) we obtain
\begin{align}\label{caffarelli}
|w_{m}(\cdot,0)|_{\mathbb{H}}\leq C
\end{align}
for any $m$ sufficiently small.\\
Finally, by using (\ref{nash}) with $v=w_{m}$ and exploiting (\ref{nablavti}) and (\ref{vti}) we have 
\begin{align}\label{nirenberg}
\limsup_{m\rightarrow 0}||w_{m}||_{L^{2}((0,T)^{N}\times (0,\delta),\xi^{1-2s})}^2 
\leq  \frac{\delta^{2-2s}}{2(1-s)}+\frac{\delta^{2}}{2s}=:C(\delta,s).
\end{align}
Taking into account (\ref{nablavti}), (\ref{caffarelli}), (\ref{nirenberg}) and by using Theorem \ref{compacttracethm}, we can deduce the existence of a subsequence, which we will denote again with  $\{w_{m}\}$, and a function $w$ such that 
\begin{align}\label{limitsforw}
&w\in L^{2}_{loc}(\mathcal{S}_{T},\xi^{1-2s}) \mbox{ and } \nabla w\in L^{2}(\mathcal{S}_{T},\xi^{1-2s});\nonumber \\
&w_{m}\rightharpoonup w \mbox{ in } L^{2}_{loc}(\mathcal{S}_{T},\xi^{1-2s}) \mbox{ as } m\rightarrow 0; \\
&\nabla w_{m}\rightharpoonup \nabla w \mbox{ in } L^{2}(\mathcal{S}_{T},\xi^{1-2s}) \mbox{ as } m\rightarrow 0; \nonumber  \\ 
&w_{m}(\cdot,0)\rightarrow w(\cdot,0) \mbox{ in } L^{q}(0,T)^{N} \mbox{ for any } q\in \Bigl[2, \frac{2N}{N-2s}\Bigr), \mbox{ as } m\rightarrow 0. \nonumber 
\end{align}
Now we distinguish two cases.
Firstly we suppose that 
\begin{equation}\label{3.245}
w\equiv 0.
\end{equation}
As in Lemma \ref{lemma6}, we define $t_{m}\in [0,1]$ such that
\begin{equation}\label{3.255}
\mathcal{J}_{m}(t_{m}v_{m})=\max_{t\in [0,1]} \mathcal{J}_{m}(t v_{m}).
\end{equation}
Set $r_{n}=2\sqrt{n}$ and note that $\displaystyle{\frac{r_{n}}{||v_{m}||_{\mathbb{X}_{m}}}\in (0,1)}$ for $m$ sufficiently small and for any $n\in \N$.
By $(f4)$, (\ref{limitsforw}) and (\ref{3.245}) easily follows that
\begin{equation}\label{3.305}
\int_{\partial^{0}\mathcal{S}_{T}}F(x,r_{n}w_{m}) dx\rightarrow 0 \mbox{ as } m\rightarrow 0
\end{equation}
for any $n\in \N$.
Then we can deduce that
$$
\mathcal{J}_{m}(t_{m}v_{m})\geq \mathcal{J}_{m}(r_{n}w_{m})\geq 2n-\int_{\partial^{0}\mathcal{S}_{T}} F(x,r_{n}w_{m}) dx\geq n
$$
provided $m$ is sufficiently small and for any $n\in \N$. 
As a consequence
\begin{equation}\label{3.315}
\mathcal{J}_{m}(t_{m}v_{m})\rightarrow \infty \mbox{ as } m\rightarrow 0.
\end{equation}
Since $\mathcal{J}_{m}(0)=0$ and $\mathcal{J}_{m}(v_{m})\in [K_{1}, K_{2}]$ we deduce that $t_{m}\in (0,1)$. Thus, by (\ref{3.255}) we have
\begin{equation}\label{3.325}
\langle{\mathcal{J}'_{m}(t_{m}v_{m}),t_{m}v_{m}}\rangle=t_{m} \frac{d}{dt}\Bigr|_{t=t_{m}} \mathcal{J}_{m}(t v_{m})=0.
\end{equation} 
Taking into account $(f6)$, (\ref{zero}), (\ref{alpham}) and (\ref{3.325}) we get
\begin{align*}
\frac{2}{\gamma}\mathcal{J}_{m}(t_{m}v_{m})&=\frac{2}{\gamma}\Bigl(\mathcal{J}_{m}(t_{m}v_{m})- \frac{1}{2}\langle{\mathcal{J}'_{m}(t_{m}v_{m}),t_{m}v_{m}}\rangle \Bigr)\\
&=\frac{1}{\gamma} \int_{\partial^{0}\mathcal{S}_{T}}G(x,t_{m}v_{m}) dx \\
&\leq \int_{\partial^{0}\mathcal{S}_{T}}G(x,v_{m}) dx \\
&=2\mathcal{J}_{m}(v_{m})-\langle{\mathcal{J}'_{m}(v_{m}),v_{m}}\rangle \leq 2 K_{2}
\end{align*}
which contradicts (\ref{3.315}).\\
Now we assume that
\begin{equation}\label{zdiv05}
w\not\equiv 0.
\end{equation}
Thus the set $\Omega=\{x\in (0,T)^{N}: w(x,0)\neq 0\}$ has positive Lebesgue measure and by using (\ref{defwm}), (\ref{limitsforw}) and (\ref{zdiv05}) we get
\begin{equation}\label{3.345}
|v_{m}(x,0)|\rightarrow \infty \mbox{ a.e. } x\in \Omega \mbox{ as } m\rightarrow 0.
\end{equation}
In particular, by $(f5)$, follows that 
\begin{align}\label{Finfty}
\frac{F(x, v_{m}(x,0))}{||v_{m}||^{2}_{\mathbb{X}_{m}}} &= \frac{F(x, v_{m}(x,0))}{|v_{m}(x,0)|^{2}} \frac{|v_{m}(x,0)|^{2}}{||v_{m}||^{2}_{\mathbb{X}_{m}}} \nonumber \\
&= \frac{F(x, v_{m}(x,0))}{|v_{m}(x,0)|^{2}} |w_{m}(x,0)|^{2} \rightarrow +\infty \mbox{ a.e. } x\in \Omega. 
\end{align}
Since 
$$
\frac{\mathcal{J}_{m}(v_{m})}{||v_{m}||^{2}_{\mathbb{X}_{m}}} \rightarrow 0 \mbox{ as } m\rightarrow 0,  
$$
using (\ref{Finfty}) and $F\geq 0$, we can deduce via the Fatou's Lemma that
\begin{equation}\begin{split}\label{3.355}
o(1)&=\frac{1}{2}-\frac{m^{2s}}{2}\frac{|v_{m}(\cdot,0)|_{L^{2}(0,T)^{N}}^{2}}{||v_{m}||^{2}_{\mathbb{X}_{m}}}-\int_{(0,T)^{N}}\frac{F(x,v_{m}(x,0))}{||v_{m}||^{2}_{\mathbb{X}_{m}}} dx  \\
&\leq \frac{1}{2}-\int_{\Omega}\frac{F(x,v_{m}(x,0))}{||v_{m}||^{2}_{\mathbb{X}_{m}}} dx \rightarrow -\infty \quad \mbox{ as } m\rightarrow 0
\end{split}
\end{equation}
that is a contradiction.
Then we can assume the  existence of  $\{v_{m}\}$ and $v$ verifying (\ref{limitsforv}).
At this point, we prove that $v$ is a weak solution to (\ref{R'}).
We proceed as in \cite{A3}.
Fix $\varphi \in \mathbb{X}$.
We know that $v_{m}$ satisfies 
\begin{equation}\label{spqr}
\iint_{\mathcal{S}_{T}} \xi^{1-2s}(\nabla v_{m} \nabla \eta+m^{2}v_{m}\eta) \; dxd\xi=\int_{\partial^{0} \mathcal{S}_{T}} [m^{2s}v_{m}+f(x,v_{m})]\eta  \; dx
\end{equation}
for every $\eta \in \mathbb{X}_{m}$.
Now, we introduce $\psi\in \mathcal{C}^{\infty}([0,\infty))$ defined as follows
\begin{equation}\label{xidef}
\left\{
\begin{array}{cc}
\psi=1 &\mbox{ if } 0\leq \xi\leq 1 \\
0\leq \psi \leq 1 &\mbox{ if } 1\leq \xi\leq 2 \\ 
\psi=0 &\mbox{ if } \xi\geq 2 
\end{array}.
\right.
\end{equation}
We set $\psi_{R}(\xi):=\psi(\frac{\xi}{R})$ for $R>1$. Then choosing $\eta=\varphi \psi_{R}\in \mathbb{X}_{m}$ in (\ref{spqr}) and taking the limit as $m\rightarrow 0$ we have 
\begin{equation}\label{limitR}
\iint_{\mathcal{S}_{T}} \xi^{1-2s}\nabla v \nabla (\varphi \psi_{R}) \; dxd\xi=\int_{\partial^{0} \mathcal{S}_{T}} f(x,v)\varphi \; dx.
\end{equation}
By passing to the limit in (\ref{limitR}) as $R\rightarrow \infty$ we deduce that $v$ verifies
$$
\iint_{\mathcal{S}_{T}} \xi^{1-2s} \nabla v \nabla \varphi \; dxd\xi-\int_{\partial^{0} \mathcal{S}_{T}} f(x,v)\varphi  \; dx=0 \quad \forall \varphi \in \mathbb{X}.
$$
Finally we show that $v$ is not identically zero. 
By using (\ref{zero}), (\ref{alpham}), $1<p<\frac{N+2s}{N-2s}$ , (\ref{f}) and (\ref{F}) with $\varepsilon=\frac{1}{4}$ we can see that
\begin{align}\label{lastpassage}
2K_{1} \leq 2\mathcal{J}_{m}(v_{m})&=2 \mathcal{J}_{m}(v_{m})- \langle\mathcal{J}'_{m}(v_{m}),v_{m}\rangle \nonumber \\
&= \int_{\partial^{0} \mathcal{S}_{T}}f(x,v_{m})v_{m}-2F(x,v_{m}) dx \nonumber \\
&\leq |v_{m}(\cdot,0)|_{L^{2}(0,T)^{N}}^{2}+(p+3)C_{\frac{1}{4}}|v_{m}(\cdot,0)|_{L^{p+1}(0,T)^{N}}^{p+1} \nonumber \\
&\leq T^{\frac{N(p-1)}{2(p+1)}}|v_{m}(\cdot,0)|_{L^{p+1}(0,T)^{N}}^{2}+(p+3)C_{\frac{1}{4}}|v_{m}(\cdot,0)|_{L^{p+1}(0,T)^{N}}^{p+1} 
\end{align}
where in the last inequality we have used H\"older inequality.
Taking the limit in (\ref{lastpassage}) as $m\rightarrow 0$ and recalling that $v_{m}(\cdot,0)$ converges strongly to $v(\cdot,0)$ in $L^{p+1}(0,T)^{N}$, we deduce that $|v(\cdot,0)|_{L^{p+1}(0,T)^{N}}>0$, that is $v\not \equiv 0$.

\begin{remark}
By exploiting the estimates (\ref{2v})-(\ref{vmfinite}), we can proceed similarly as in \cite{A3} to 
infer that $v(\cdot,0)\in \mathcal{C}^{0,\alpha}([0,T]^{N})$ for some $\alpha \in (0,1)$.
\end{remark}

\end{document}